\documentclass[a4paper]{article}

\usepackage{amsmath, amsthm, amsfonts, amssymb}

\usepackage{proof}

\usepackage{fancyvrb}

\usepackage{verbatim}

\usepackage{graphicx}
\usepackage{epstopdf}

\usepackage{hyperref}
\hypersetup{
	colorlinks,
	citecolor=black,
	filecolor=black,
	linkcolor=black,
	urlcolor=blue
}

\usepackage{fancyhdr}

\newtheorem{theorem}{Theorem}[section]
\newtheorem{lemma}[theorem]{Lemma}
\newtheorem{conjecture}[theorem]{Conjecture}
\newtheorem{corollary}[theorem]{Corollary}

\theoremstyle{definition}
\newtheorem{remark}[theorem]{Remark}

\newtheorem*{acknowledgement}{Acknowledgement}

\usepackage{amssymb}
\usepackage{amsmath}
\usepackage{eufrak}

\theoremstyle{remark}

\numberwithin{equation}{section}

\title{Chen primes in arithmetic progressions}
\date{}

\author{Pawe\l  ~Lewulis}

\begin{document}

\maketitle


\begin{abstract} 
We find a lower bound for the number of Chen primes in the arithmetic progression $a \bmod q$, where $(a,q)=(a+2,q)=1$. Our estimate is uniform for $q \leq \log^M x$, where $M>0$ is fixed.
\end{abstract}
\[ \]
\section{Introduction}
\

The famous twin prime conjecture asserts that there exist infinitely many primes $p$ such that $p+2$ is also a prime. It is a common perception that our current methods are insufficient to prove this supposition, although in some way one is able to get very close to it. For example in \cite{A} and \cite{B}, J. Chen proved the following famous result.
\text{ \\ }

\begin{theorem}\label{Chen} There are infinitely many primes $p$ such that $p+2$ has at most two prime factors, each greater or equal than $p^{1/10}$.
\end{theorem}

In fact, Chen showed that the number of such primes in an interval $[1,x]$ is greater than $cx/\log^2 x$, where $c$ is some positive constant. In \cite{C}, B. Green and T. Tao showed that there are infinitely many 3-term arithmetic progressions among them. This result was extended by B. Zhou in \cite{D}, where it is proven that in the same circumstances one can find arithmetic progressions of any given length.

It is natural to ask whether Chen primes (or twin primes) are equally distributed among arithmetic progressions as primes do. For example, using Cram\'er's random model we are led to the following conjecture about twin primes (cf. section 2 for the notation).

\begin{conjecture}\label{1.1} Let $a,q$ be positive integers such that $(a(a+2),q)=1$. Then
\[ \# \{ \mbox{\textup{twin primes} $p$ : $p \equiv a ~(\bmod ~q)$}\} = (2 \Pi_2 + o_q(1)) \frac{ x}{\varphi_2 (q) \log^2 x}. \]
\end{conjecture}

The goal of this paper is to prove a uniform (in a Siegel$-$Walfisz manner) lower bound for an analogous count of Chen primes. We will accomplish this using techniques developed by Chen with slight modifications (specifically, we shall follow the discussion in \cite{BLOG}). The main result can be stated as follows:

\begin{theorem}\label{1.3} Let $M>0$ and let $a,q$ be positive integers such that $(a,q)= (a+2,q)=1$ and $q \leq \log^M x$. Then
\[ \displaystyle  \sum_{x/2 \leq n \leq x-2} \Lambda_{a,q}(n) \mathbf{1}_{{\mathcal P}_2}(n+2) \mathbf{1}_{(n+2,P(x^{1/8}))=1} \gg_M \frac{x}{\varphi_2 (q) \log x}. \]
\end{theorem}
After removing the small contribution of the prime powers, we can also convert this into a more elegant form.

\begin{corollary}\label{1.2}  Let $M>0$ and let $a,q$ be positive integers such that $(a,q)= (a+2,q)=1$ and $q \leq \log^M x$. Then
\[  \sum_{ \substack{ n \leqslant x \\ n \equiv a ~(q) }}   \mathbf{1}_{n \in \mathcal{P}_{Ch}} \gg_M \frac{x}{\varphi_2 (q) \log^2 x}. \]
\end{corollary}

The proof of Theorem \ref{1.3} is based on the following observation. 

\begin{lemma}\label{1.4} For $x^{2/3}<n \leqslant x$ we have
\[ \displaystyle   \mathbf{1}_{{\mathcal P}_2}(n) \geq 1 - \frac{1}{2} \sum_{p \leq x^{1/3}}  \mathbf{1}_{p|n} - \frac{1}{2} \sum_{p_1 \leq x^{1/3} < p_2 \leq p_3}  \mathbf{1}_{n=p_1p_2p_3} - \sum_{p \leq x^{1/3}} \mathbf{1}_{p^2|n}. \]
\end{lemma}

\begin{proof} This is an easy case analysis.
\end{proof}

By Lemma \ref{1.4} one can reduce our task into the following estimation.

\begin{equation}\label{esti}
 A_1 - \frac{1}{2} \sum_{x^{1/8} \leq p \leq x^{1/3}} A_{2,p} - \frac{1}{2} A_3 -
\sum_{x^{1/8} \leq p \leq x^{1/3}} A_{4,p} \gg \frac{x}{\varphi_2 (q) \log x}, 
\end{equation}
where

\begin{align}\label{sumki} A_1 &:= \sum_{x/2 \leq n \leq x-2} \Lambda_{a,q}(n) \mathbf{1}_{(n+2,P(x^{1/8}))=1}, \\
A_{2,p} &:= \sum_{\substack{ x/2 \leq n \leq x-2 \\ p|n+2}} \Lambda_{a,q}(n) \mathbf{1}_{(n+2,P(x^{1/8}))=1},  \\
A_3 &:= \sum_{x/2 \leq n \leq x-2} \Lambda_{a,q}(n) \sum_{x^{1/8} \leq p_1 \leq x^{1/3} < p_2 \leq p_3}  \mathbf{1}_{n+2=p_1p_2p_3}. \\
A_{4,p} &:= \sum_{\substack{ x/2 \leq n \leq x-2 \\ p^2|n+2}} \Lambda_{a,q}(n) \mathbf{1}_{(n+2,P(x^{1/8}))=1}, 
\end{align}
To this end, we prove the following estimates.
\begin{align} A_1 & \geq (4.394 -o(1)) \Pi_2 \frac{x}{2\varphi_2 (q) \log x}~~~~ \mbox{(section 4)}, \\
\sum_{x^{1/8} \leq p \leq x^{1/3}}  A_{2,p} &\leq  (7.168 + o(1)) \Pi_2 \frac{x}{2\varphi_2 (q) \log x}~~~~  \mbox{(section 5)}, \\
A_3 & \leq  (1.456 + o(1)) \Pi_2 \frac{x}{2 \varphi_2 (q) \log x}~~~~  \mbox{(section 6)}.
\end{align}
Notice that the right hand side sum in (\ref{esti}) can be handled easily. Indeed, we have
\begin{equation}
\sum_{x^{1/8} \leq p \leq x^{1/3}}  A_{4,p} \ll x \log x  \sum_{x^{1/8} \leq p \leq x^{1/3}} \frac{1}{p^2} \ll x^{7/8+o(1)}. 
\end{equation}

\begin{acknowledgement} This work is heavily influenced by the article \cite{BLOG} from T. Tao's blog for which I am very grateful. I would also like to thank M. Radziejewski, J. Kaczorowski and P. Achinger for valuable discussions and many corrections.
\end{acknowledgement}

\section{Notation}

\begin{itemize}
\item $p$ always denotes a prime number and $x$ will be a real number greater than 3;
\item $P(y):= \prod_{p<y}p$ for $y \geq 2$,
\item $\Lambda_{a,q} (n) := \mathbf{1}_{n \equiv a (q)} \Lambda(n)$,
\item $\mathcal{P}$, $\mathcal{P}_2$ are the sets of primes and almost primes respectively,
\item  $\mathcal{P}_{Ch} := \{ p : \exists_{p_1, p_2 \in \mathcal{P}}~p_1,p_2 \geqslant p^{1/8},~ p+2=p_1p_2 \}~ \cup ~ \{ \mbox{twin primes} \}, $
\item  ${\Pi_2 := \prod_{p>2} (1-\frac{1}{(p-1)^2})} \approx 0.66$,
\item $\varphi_2 (n) := n \prod_{p|n:p\not=2} \left( 1 - \frac{2}{p} \right) \times (1 - \mathbf{1}_{2|n}/2 )$,
\item $\Delta ( f; a~ (d) ) := \sum_{n \equiv a (d)} f(n) - \frac{1}{\phi(d)} \sum_{n: (n,d)=1 } f(n)$,
\item we will treat $M$ as an absolute constant, so in the further sections the dependence on $M$ will not be emphasized in any way.
\end{itemize}

\section{Preliminaries}

\begin{theorem}[Siegel$-$Walfisz]\label{SW}
For any $A > 0$ there exists a constant $C_A>0$ such that

\[ \sum_{ \substack{n \leq x \\ n=a\ (d)}} \Lambda(n) = \frac{x}{\phi(d)} + O( x \exp( - C_A \sqrt{\log x} ) ) \]
for every residue class $a\ (d)$ satisfying $(a,d)=1$ and for all $x \geq 2$ such that $d \leq \log^A x$.
\end{theorem}

\begin{theorem}[Bombieri$-$Vinogradov]\label{BV} Let $x \geq 2$. Then

\[  \sum_{d \leq D} \max_{a \in ({\bf Z}/d{\bf Z})^\times} 
\left| \Delta ( \Lambda \mathbf{1}_{[1,x]} ; a ~(d) ) \right|
 \ll_{A} x \log^{-A} x\]
for all $A > 0$, where $D \leq x^{1/2} \log^{-B} x$ for some sufficiently big $B = B(A)$.
\end{theorem}

Let us define the real functions $f(s)$ and $F(s)$ in the following way: 

\[  F(s) = \frac{2e^\gamma}{s}, ~~~~ f(s)= \frac{2e^\gamma}{s} \log(s-1). \]

\begin{theorem}[Jurkat$-$Richert]\label{LS}

  Consider $s>1$ and $z, D \geq 2$ which satisfy $z = D^{1/s}$. Let $\mathcal{Q}$ be a finite subset of $\mathcal{P}$ and let $Q$ be the product of the primes in $\mathcal{Q}$. Furthermore, let $h$ be a multiplicative function that for some $ \varepsilon$ with $0< \varepsilon < 1/200$ satisfies the inequality

\begin{equation}
\prod_{\substack{ p \in \mathcal{P}\diagdown \mathcal{Q} \\ u \leq p < z}} (1 - h(p))^{-1} < (1 + \varepsilon ) \frac{\log z}{\log u}  
\end{equation}
and
\[ \displaystyle  0 \leq h(p) < 1. \]
for every prime $p$. For a fixed integer $\alpha$, let $\left\{ E_d \right\}_{d=1}^\infty$ be a family of subsets of $\mathbf{Z}^+$ of the form $\{n: n \equiv \alpha \bmod d\}$. Let us consider a finitely supported non-negative real sequence $(a_n)_{n=1}^\infty$ and let $r_d$ be defined by the equality
\[ \displaystyle  \sum_{n \in E_d} a_n = h(d) X + r_d. \]
for every square-free $d \leq D$, some $X>0$ and some error terms $r_d$. Then, for any $1 < s \leq 3$, there is an upper bound

\begin{equation}
  \sum_{n\not \in\bigcup_{p <z} E_p} a_n \leq (F(s) +  \varepsilon e^{14-s} ) X V(z) +  \sum_{d \leq QD: \mu^2(d)=1} |r_d| ~~+ ~| \alpha |, 
  \end{equation}
and for any $2 \leq s \leq 4$ there is a lower bound

\begin{equation}
\sum_{n\not \in\bigcup_{p <z} E_p} a_n \geq (f(s) -  \varepsilon e^{14-s} ) X V(z) + \sum_{d \leq QD: \mu^2(d)=1} |r_d| ~~- ~ | \alpha |, 
\end{equation}
where $V(z)=\prod_{p<z}(1-h(p))$.

\end{theorem}

\begin{proof}
This is a special case of the Jurkat-Richert theorem from \cite{NATH}.
\end{proof}

\begin{remark}\label{Rem}
In \cite[\S10.3]{NATH} one can also find a proof that our particular choice of $h$ in the next section satisfies condition $(3.1)$ with $\mathcal{Q}$ being the set of first $y$ primes where $y$ depends on $\varepsilon$. In this situation we can also get $\varepsilon \rightarrow 0$ by letting $Q \rightarrow \infty$.
\end{remark}

\begin{lemma}\label{2.4} Let
\[  b(n) = \mathbf{1}_{x/2+2 \leq n \leq x} \sum_{x^{1/8} \leq p_1 \leq x^{1/3} < p_2 \leq p_3} \mathbf{1}_{n=p_1p_2p_3}.\]
Then

\[  \sum_{d \leq D} \max_{a \in ({\bf Z}/d{\bf Z})^\times} |R_{a,d}| :=  \sum_{d \leq D} \max_{a \in ({\bf Z}/d{\bf Z})^\times} 
\left| \sum_{n \equiv a \bmod d} b(n) - \frac{1}{\varphi (d)} \sum_{n } b(n) \right|
 \ll_{A} x \log^{-A} x\]
for any $A>0$ provided that $D \leq x^{1/2} \log^{-B} x$ for some sufficiently large $B = B(A)$. 
\end{lemma}

\begin{proof}
Follows from Theorem 22.3 from \cite{FI}. We also note that for each integer $d \geq x^{1/8}$ we have $\sum_{n:(n,d)>1} b(n) \ll x^{7/8}$, so the total contribution of ingredients of this form is at most $x^{7/8+o(1)}$ which is neglible. 
\end{proof}

\begin{lemma}\label{SUMY}
Let $d \geq 1$ be a fixed integer and let $f: (0,\infty)^d \rightarrow \bf{C}$ be a fixed compactly supported, Riemann integrable function. Then for $x>1$ we have
\end{lemma}

\[ \sum_{p_1, \dots , p_d} \frac{1}{p_1 \dots p_d} f \left( \frac{ \log p_1}{\log x}, \dots , \frac{ \log p_d}{\log x} \right)  =
\int_{(0,\infty)^d} f(t_1, \dots , t_d ) \frac{ dt_1 \dots dt_d}{t_1 \dots t_d} + o_{x \rightarrow \infty} (1).\]
\text{ \\ }
\begin{proof} Follows from Mertens' second theorem combined with elementary properties of the Riemann integral.
\end{proof}
{ \ }

\section{Estimating $A_1$}
{ \ }

According to Theorem \ref{LS} let us define $E_d$ as the set of positive integers from the residue class $-2 ~(d)$. Let us also define a multiplicative function $g(d)$ such that $g(2)=0$ and $g(p)=1/(p-1)$. The precise value in non-square-free numbers is not relevant. Put $z=x^{1/8}$ and $D=x^{1/2-\epsilon_x}$, where $\epsilon_x:=(\log x)^{-1/2}$. We take $\mathcal{Q}$ to be the largest set of the form as in Remark \ref{Rem} satisfying $Q \leq \log \log x$. We have

\[ A_1 =  \sum_{\substack{ x/2 \leq n \leq x-2 \\ n \not\in \bigcup_{p<z} E_p}} \Lambda_{a,q}(n). \]

We define

\[  \sum_{\substack{ x/2 \leq n \leq x-2 \\ n \in E_d}} \Lambda_{a,q}(n) = 
\sum_{\substack{ x/2 \leq n \leq x-2 \\ n \equiv -2 ~(d) \\ n \equiv a ~(q)}} \Lambda (n) = g_{a,q}(d) \frac{x}{2} + r_d^{(a,q)},\]
where $g_{a,q}(d):= \mathbf{1}_{(q,d)=1}  g(d)/ \varphi (q)$ and $r_d^{(a,q)}$ is a remainder term. Here we notice that if $p|d,q$ for some  $p$, then $p|n+2$ which contradicts $n+2 \equiv a+2~ (q)$ joined with $(a+2,q)=1$. In such a case one has $ r_d^{(a,q)} = 0$.

According to Theorem \ref{LS}, we put $X=\frac{x}{2\varphi(q)}$ and $h(d)= g(d) \mathbf{1}_{(q, d)}$. If $(q,d)=1$, then  by the prime number theorem we get

\[  r_d^{(a,q)} = \Delta ( \Lambda \mathbf{1}_{[x/2,x-2]} ; c_{q,d} ~(qd) ) + O \left( \frac{x}{\varphi (qd) \exp( C \log^{1/10} x )} \right) .\]
for some residue class $c_{q,d} \in ({\bf Z}/ qd{\bf Z})^\times$ and some positive constant $C$. 

From the Bombieri$-$Vinogradov theorem one gets

\begin{equation}\label{32}
 \sum_{\substack{ d \leq QD \\ \mu^2 (d)=1}} |  \Delta ( \Lambda \mathbf{1}_{[x/2,x-2]} ; c_{q,d} ~(qd) )| \leq 
\sum_{d \leq qQD} \max_{c \in  ({\bf Z}/ d{\bf Z})^\times} | \Delta ( \Lambda \mathbf{1}_{[x/2,x-2]} ; c ~(d) )| \ll  x \log^{-M-10} x. \end{equation}
Notice that the estimate above is uniform because

\begin{equation}\label{33}
qQD \ll x^{1/2}x^{-1/(\log x)^{1/2}} \log^M x \log \log x = o\left( x^{1/2} \log^{-B(M+10)}x \right),
\end{equation}
where $B(M+10)$ is a real number large enough such that Theorem \ref{BV} can be used with exponent $M+10$.

By Theorem \ref{LS} one has

\begin{equation}\label{34}
A_1 \geq (f(4-8\epsilon_x) - O(\varepsilon)) \frac{x}{2 \varphi (q)} V(z) - O( x \log^{-M-10} x ). 
\end{equation}
From Mertens' third theorem we conclude

\begin{equation}\label{35}
V(z) = (1+o(1))\frac{2 \Pi_2}{e^\gamma \log z} \prod_{p \not=2 : p|q} \left( 1 - \frac{1}{p-1} \right)^{-1}.
\end{equation}

This can be simplified further to

\begin{equation}\label{36}
A_1 \geq (\log 3-o(1)) \Pi_2 \frac{x}{2\varphi_2 (q) \log z}.
\end{equation}

\section{Estimating $A_{2,p}$}
{ \ }

For $z \leq p \leq x^{1/3}$ we have

\[  A_{2,p} =  \sum_{\substack{ x/2 \leq n \leq x-2 \\ n \not\in \bigcup_{p'<z} E_{p'}}} \Lambda_{a,q}(n)\mathbf{1}_{p|n+2}. \]
We apply the Jurkat$-$Richert theorem, but this time taking $D/p$ instead of $D$. For any square-free $d$ we have

\begin{equation}\label{41}
\sum_{\substack{ x/2 \leq n \leq x-2 \\ n \in E_d}} \Lambda_{a,q}(n)\mathbf{1}_{p|n+2} =  
\sum_{\substack{ x/2 \leq n \leq x-2 \\ n \equiv -2 ~(pd) \\ n \equiv a ~(q) }} \Lambda(n) =
 g(d)\mathbf{1}_{(d,q)=1} \frac{g(p)x}{2\varphi(q)} + r_{pd}^{(a,q)},  
\end{equation}
and hence

\begin{equation}\label{42}
A_{2,p} \leq \left( F\left( \frac{\log D/p}{\log z} \right) + O(\varepsilon) \right) \frac{g(p)x}{2 \varphi (q)} V(z) + O\left( \sum_{d \leq QD/p} | \Delta ( \Lambda \mathbf{1}_{[x/2,x-2]} ; c_{q,pd} ~(pqd) ) | \right). 
\end{equation}
Since every $d \leq QD$ has at most $O(\log x)$ prime factors, one can conclude

\begin{equation}\label{43}
 \sum_{z \leq p \leq x^{1/3}} \sum_{d \leq QD/p} | \Delta ( \Lambda \mathbf{1}_{[x/2,x-2]} ; c_{q,pd} ~(pqd) ) | \ll \log x  \sum_{d \leq QD} |\Delta ( \Lambda \mathbf{1}_{[x/2,x-2]} ; c_{q,d} ~(qd) )| \ll x \log^{-M-9}x. 
\end{equation}

From (\ref{35}), (\ref{42}) and (\ref{43})

\begin{equation}\label{44}
\sum_{z \leq p \leq x^{1/3}}  A_{2,p} \leq
\Pi_2 \frac{x}{ e^{\gamma}\varphi_2 (q) \log z}  \sum_{z \leq p \leq x^{1/3}}  \frac{ \left( F\left( \frac{\log D/p}{\log z} \right) +O(\varepsilon ) \right)}{p} + O(x \log^{-M-9}x). 
\end{equation}
By Lemma 2.5 we have

\begin{equation}\label{45}
\sum_{z \leq p< x^{1/3}} \frac{F( \frac{\log D/p}{\log z} )}{p} = \int_1^{8/3} F( 4-8\epsilon_x -t ) \frac{dt}{t} + o(1) =
2e^{\gamma} \int_1^{8/3} \frac{1}{4-t-8\epsilon_x} \frac{dt}{t} + o(1) = 
\end{equation}
\[ (2e^{\gamma}+o(1))  \int_1^{8/3} \frac{dt}{(4-t)t} = \frac{e^\gamma \log 6}{2} + o(1).\]

Combining (\ref{44}) and (\ref{45}) we get

\begin{equation}\label{46}
 \sum_{z \leq p \leq x^{1/3}}  A_{2,p} \leq  (\log 6 + o(1)) \Pi_2 \frac{x}{2\varphi_2 (q) \log z}. 
 \end{equation}

\section{Estimating $A_3$}
{ \ }

We have

\begin{equation}\label{51}
A_3 =  \sum_{x/2 \leq n \leq x-2} \Lambda_{a,q}(n) \sum_{z \leq p_1 \leq x^{1/3} < p_2 \leq p_3} \mathbf{1}_{n+2=p_1p_2p_3} = 
\end{equation}
\[ \sum_{\substack{ x/2 + 2 \leq n \leq x \\ n \equiv a+2 ~(q)}} \Lambda(n-2) \sum_{x^{1/8} \leq p_1 \leq x^{1/3} < p_2 \leq p_3} \mathbf{1}_{n=p_1p_2p_3} \leq \log x \sum_{ \substack{n \in {\bf Z} \\ n \equiv a+2 ~(q)}} b(n) \mathbf{1}_{(n-2,P(\sqrt{x}))=1} + O(x^{0.51}),\]
where $b(n)$ is defined as in Lemma \ref{2.4}. The error term comes from the numbers of the form $p^k$ for $k \geq 2$. One can rewrite the above sum into

\[  \sum_{\substack{ n \in {\bf Z} \\ n \not\in \bigcup_{p<\sqrt{x}} E'_p}} b(n) \mathbf{1}_{n \equiv a+2 ~(q)}, \]
where $E'_d$ denotes the residue class $2~(d)$. We use the Jurkat$-$Richert theorem again. By Lemma \ref{2.4} for some $c'_{d,q} \in ({\bf Z}/dq{\bf Z})^\times$ we get (remember that conditions $n \equiv 2 ~(d)$ and $n \equiv a+2 ~(q)$ force $(d,q)=1$)

\begin{equation}\label{52}
  \sum_{n \in E'_d} b(n) \mathbf{1}_{n \equiv a+2 ~(q)} =    \sum_{\substack{ n \equiv 2 ~(d) \\ n \equiv a+2 ~(q)}} b(n) =
\frac{g(d)}{ \varphi(q)} \mathbf{1}_{(d,q)=1} \sum_{ n \in {\bf Z}} b(n) + R_{c'_{d,q},dq} =
\end{equation}
\[ g(d)  \mathbf{1}_{(d,q)=1} \sum_{ n \equiv a ~(q)} b(n) - g(d) \mathbf{1}_{(d,q)=1} R_{a,q} + R_{c'_{d,q},dq}, \]
where
\[ \sum_{\substack{ d \leq QD \\ \mu^2 (d)=1}} | R_{c'_{d,q},dq} -  g(d)\mathbf{1}_{(d,q)=1}R_{a,q} | \leq
  \sum_{\substack{ d \leq qQD \\ \mu^2 (d)=1}} \max_{c' \in ({\bf Z}/d{\bf Z})^\times } |R_{c',d}| ~+  \] 
  \[|R_{a,q}| \sum_{\substack{ d \leq QD \\ \mu^2 (d)=1}} \frac{1}{\varphi (d)} \ll
  x \log^{-M-10} x;\]
we used trivial inequality $|R_{a,q}| \leq \sum_{d \leq D} |R_{a,d}| \ll x \log^{-M-11}x$ above. Similarily to the previous situation of this kind, let us emphasize that the upper bound here is uniform with respect to $q$. By using the inequality from Theorem \ref{LS} with the level of distribution $QD^{1/(1+\epsilon_x)}$ we get

\begin{equation}\label{zapomnialem!}
\sum_{ \substack{n \in {\bf Z} \\ n \equiv a+2 ~(q)}} b(n) \mathbf{1}_{(n-2,P(\sqrt{x}))=1} \leq (F(1+\epsilon_x)+O(\varepsilon)) V( D^{1/(1+\varepsilon_x)} )  \sum_ {\substack{n \in {\bf Z} \\ n \equiv a+2 ~(q)}} b(n) + O( x \log^{-M-10} x ). 
\end{equation}
Again, by using the Mertens' third theorem one gets

\begin{equation}\label{53}
 V(D^{1/(1+\epsilon_x)} ) = \frac{1}{2} \Pi_2 (1+ o(1)) \frac{1}{e^\gamma \log z} \prod_{p\not= 2 : p|q} \left( 1 - \frac{1}{p-1} \right)^{-1}.
\end{equation}
By $F(s)   \xrightarrow{s \rightarrow 1^+} 2e^{\gamma}$ and

\[  \sum_ {\substack{n \in {\bf Z} \\ n \equiv a+2 ~(q)}} b(n) = \frac{1}{\varphi (q)} \sum_ {n \in {\bf Z} } b(n) + R_{a,q},\]
where  $|R_{a,q}| \ll x \log^{-M-11}x$, we have

\begin{equation}\label{54}
\sum_{ \substack{n \in {\bf Z} \\ n \equiv a+2 ~(q)}} b(n) \mathbf{1}_{(n-2,P(\sqrt{x}))=1} \leq (1 + o(1)) \Pi_2 \frac{1}{\varphi_2 (q) \log z}  \sum_ {n \in {\bf Z}} b(n) + O( x \log^{-M-10} x ). 
\end{equation}
By (\ref{51})$-$(\ref{54}) and Lemma \ref{SUMY} we get

\begin{equation}\label{55}
\sum_ {n \in {\bf Z}} b(n)  \leq (1+o(1)) \frac{x}{2\log x} \int_{1/8 \leq t_1 \leq 1/3 < t_2 < 1-t_1-t_2} \frac{dt_1 dt_2}{t_1 t_2 (1-t_1-t_2)} \leq  (0.364+o(1)) \frac{x}{2\log x}. 
\end{equation}

\bibliographystyle{amsplain}

\[ \]
\textsc{Institute of Mathematics, Polish Academy of Sciences, \'{S}niadeckich 8, 00-656 Warsaw, Poland \\ \\}
\textit{E-mail address:} \texttt{pkarasek{@}impan.pl}

\end{document}